\documentclass[11pt]{amsart}
\usepackage{graphicx,color}
\usepackage{amssymb}
\usepackage{epstopdf, xypic}
\usepackage{booktabs}
\usepackage{hyperref}

\newtheorem{thm}{Theorem}[section]
\newtheorem{prop}[thm]{Proposition}
\newtheorem{lem}[thm]{Lemma}

\newtheorem{qst}[thm]{Question}

\theoremstyle{definition}

\newtheorem{eg}[thm]{Example}

\theoremstyle{remark}

\renewcommand{\xto}[1]{\overset{#1}{\lto}}

\newcommand{\N}{\mathbb N}
\newcommand{\Z}{\mathbb Z}

\renewcommand{\ker}{\mathop{\mathrm{Ker}}\nolimits}
\newcommand{\coker}{\mathop{\mathrm{Coker}}\nolimits}

\newcommand{\into}{\hookrightarrow}

\renewcommand{\tilde}{\widetilde}

\newcommand{\reg}{\mathop{\mathrm{reg}}\nolimits}
\newcommand{\codim}{\mathop{\mathrm{codim}}\nolimits}
\newcommand{\type}{\mathop{\mathrm{type}}\nolimits}

\newcommand{\hilb}{\mathop{\mathrm{Hilb}}\nolimits}

\newcommand{\ch}{\mathop{\mathrm{char}}\nolimits}

\newcommand{\Syz}{\mathop{\mathrm{Syz}}\nolimits}

\newcommand{\tensor}{\otimes}
\newcommand{\tor}{\mathop{\mathrm{Tor}}\nolimits}
\renewcommand{\hom}{\mathop{\mathrm{Hom}}\nolimits}
\newcommand{\ext}{\mathop{\mathrm{Ext}}\nolimits}

\renewcommand{\lto}{\mathop{\longrightarrow\,}\limits}

\begin{document}

\title[Quadratic Gorenstein algebras]{Quadratic Gorenstein algebras with many surprising properties}
\author[J. McCullough]{Jason McCullough}
\address{Department of Mathematics, Iowa State University, Ames, IA 50011}
\email{jmccullo@iastate.edu}

\author[A. Seceleanu]{Alexandra Seceleanu}
\address{Department of Mathematics, University of Nebraska, Lincoln, NE 68588}
\email{aseceleanu@unl.edu}
%\date{}                                           % Activate to display a given date or no date

\begin{abstract} Let $k$ be a field of characteristic $0$.  Using the method of idealization, we show that there is a non-Koszul, quadratic, Artinian, Gorenstein, standard graded $k$-algebra of regularity $3$ and codimension $8$, answering a question of Mastroeni, Schenck, and Stillman.  We also show that this example is minimal in the sense that no other idealization that is non-Koszul, quadratic, Artinian, Gorenstein algebra, with regularity $3$ has smaller codimension.

We also construct an infinite family of graded, quadratic, Artinian, Gorenstein algebras $A_m$, indexed by an integer $m \ge 2$, with the following properties: (1) there are minimal first syzygies of the defining ideal in degree $m+2$, (2) for $m \ge 3$, $A_m$ is not Koszul, (3) for $m \ge 7$, the Hilbert function of $A_m$ is not unimodal, and thus (4) for $m \ge 7$, $A_m$ does not satisfy the weak or strong Lefschetz properties.  In particular, the subadditivity property fails for quadratic Gorenstein ideals.

Finally, we show that the idealization of a construction of Roos yields non-Koszul quadratic Gorenstein algebras such that the residue field $k$ has a linear resolution for precisely $\alpha$ steps for any integer $\alpha \ge 2$.  Thus there is no finite test for the Koszul property even for quadratic Gorenstein algebras.

\end{abstract}


\maketitle
 
 \section{Introduction}
 
 Let $k$ be a field and let $S = k[x_1,\ldots,x_c]$ be a standard graded polynomial ring over $k$.  Consider $R = S/I$ an Artinian standard graded quotient of $S$.  A recent problem that has attracted some attention is to identify what conditions on a quadratic Gorenstein algebra $R$ force $R$ to be Koszul.  While every such $R$ with $\reg(R) \le 2$ is Koszul \cite{CRV}, Mastroeni, Schenck, and Stillman \cite{MSS1} constructed quadratic, Gorenstein algebras with $r=\reg(R) = 3$ and $\codim(R) = c$ for all $c \ge 9$; this negatively answered a question of Conca, Rossi, and Valla \cite{CRV}.  Mastroeni et. al. pose the following question: For which positive integers $(r,c)$ does there exist a non-Koszul, quadratic, Gorenstein algebra with regularity $r$ and codimension $c$?  In a second paper \cite{MSS2}, they settle the question in all cases except three, namely $(r,c) = (3,6), (3,7),$ and $(3,8)$.  The first result of this paper, given in Section~\ref{s3},  is to settle the case $(r,c) = (3,8)$ by finding a non-Koszul, quadratic, Gorenstein algebra with these parameters. We do so by applying Nagata's idealization construction to a non-Koszul Artinian algebra of codimension $4$, which comes from Roos' list of quadratic algebras in four variables \cite{Roos}.  Moreover, we show that the other two cases  $(r,c) = (3,6), (3,7)$ cannot be similarly settled via idealizations.
 
 Our second construction addresses the subadditivity property of Gorenstein ideals.  Set $t_i(R) = \sup \{ j\,|\,\tor_i^S(R,k)_j \neq 0\}$.  The numbers $t_i(R)$ measure the maximal degree of a minimal generator of the $i$-th syzygy module of $R$ as an $S$-module.  They are primarily of interest because of their relation to regularity as $\reg(S/I) = \max_{i \ge 0}\{t_i(S/I) - i\}$.  The ring $R$ is said to satisfy the {\em subadditivity property} if $t_a(R) + t_b(R) \ge t_{a+b}(R)$ for all $a, b \ge 1$.  It is easy to see that complete intersections satisfy subadditivity (Proposition \ref{prop:CIsub}) while general Cohen-Macaulay ideals do not (cf. \cite[Example 4.4]{EHU}).  Several recent papers have studied the subadditivity property for various classes of ideals \cite{ACI1, F, HS}. It is conjectured that monomial ideals and Koszul ideals satisfy subadditivity \cite[Conjecture 6.4]{ACI2}.  To our knowledge, there were no known counterexamples to subadditivity for Gorenstein ideals; some positive results for Gorenstein ideals are proved in \cite{ES}.  In Section~\ref{s4}, we show that subadditivity fails in a strong way for quadratic Gorenstein ideals.  As a consequence of our methods, we obtain an infinite family of quadratic Gorenstein ideals that are non-Koszul, have arbitrarily high degree first syzygies, have non-unimodal Hilbert function, and do not satisfy the strong or weak Lefschetz properties.  This provides a counterexample to a conjecture of Migliore and Nagel \cite[Conjecture 4.5]{MN}; an earlier counterexample was given by Gondim and Zappala \cite{GZ}.  
 
 The third construction modifies a separate example of Roos \cite{Roos2} to show that there is no finite test of the Koszul property even for quadratic Gorenstein algebras.  In Section~\ref{s5}, we show that for any integer $\alpha \ge 2$, there is a quadratic, Artinian, Gorenstein $k$-algebra $B_\alpha$ with $\codim(B_\alpha) = 14$ and $\reg(B_\alpha) = 3$ such that the resolution of $k$ as a $B$-module is linear for precisely $\alpha$ many steps.

\section{Background}
 
 Here we collect notation and results needed in the rest of the paper.  Let $k$ be a field, $S = k[x_1,\ldots,x_n]$ a standard graded polynomial ring over k, and $R = S/I$, where $I$ is a homogeneous ideal of $S$.  Then $R$ inherits a decomposition $R = \oplus_{i \ge 0} R_i$ as $K$-vector spaces with the property that $R_i \cdot R_j \subseteq R_{i+j}$.  The {\em Hilbert function} of $R$ is $\mathrm{HF}_R(i)=\dim_k(R_i)$.  If $\mathrm{HF}_R(i) = 0$ for $i \gg 0$, $R$ is Artinian; this is equivalent to requiring $\dim_k(R) < \infty$ or that $R$ satisfies the descending chain condition on ideals.  The generating function for the Hilbert function is the {\em Hilbert series} of $R$ defined as $\mathrm{HS}_R(t) = \sum_{i} \dim_k(R_i)t^i$ and similarly for a graded $R$-module.  For a graded Artinian  ring $R$, the {\em $h$-vector} records the nonzero values of the Hilbert function of $R$. The syzygy modules of $R$ are denoted $\Syz_i^S(R)$. The {\em regularity} of $R$ is $\reg(R) = \max\{j\,|\,\beta_{ij}^S(R) \neq 0\}$, where $\beta_{ij}^S(R) = \tor_i^S(R,k)_j$ are the graded Betti numbers of $R$ over $S$.  Regularity is one of the most well-studied invariants of graded $k$-algebras and has connections to sheaf cohomology and computational complexity \cite{BM}.  In particular, if $R = S/I$ as above, $\reg_S(S/I) + 1$ is an upper bound on the degrees of a minimal generating set of $I$; however, there are examples showing that $\reg_S(S/I)$ can be doubly exponential in the degrees of the generators of $I$ \cite{Ko}.
 
  The ring $R$ is called {\em Koszul} if $k$ has a linear free resolution over $R$; that is, $\beta^R_{ij}(k) = 0$ for all $j > i$.   It is well-known that Koszul algebras are defined by quadratic ideals and that ideals having a Gr\"obner basis of quadrics define Koszul algebras, but both of these implications are irreversible \cite[Remark 1.10 and Example 1.20]{Co}.  Every quadratic
complete intersection (that is, rings of the form $S/(f_1,\ldots,f_m)$, where $f_1,\ldots,f_m$ is a graded regular sequence on $S$) is Koszul by a result of Tate \cite{Tate}.  There are many examples of Koszul algebras in algebraic geometry and these algebras enjoy a rich duality theory.  The article \cite{Conca} contains a modern introduction to the theory of (commutative) Koszul algebras.

A graded Artinian $k$-algebra $R$ is said to satisfy the {\em weak Lefschetz property} if there is a linear form $\ell\in R_1$ such that for each non negative integer $i$ the $k$-linear map $R_i\to R_{i+1}, r\mapsto \ell r$ is either injective or surjective. Similarly,  $R$ is said to satisfy the {\em strong Lefschetz property} if there is a linear form $\ell\in R_1$ such that for each pair of non negative integers $i,j$ the $k$-linear map $R_i\to R_{i+j}, r\mapsto \ell^j r$ is either injective or surjective.  Lefshetz properties of Artinian $k$-algebras have been well-studied and we refer the reader to \cite{Lef} or \cite{MN2} for an overview of the area.
 
  If $R$ is graded Artinian, then the canonical module of $R$ is given by $\omega_R = \ext_S^n(R,S)(-n)$ and the canonical module of an $R$-module $M$ is given by $\omega_M = \ext_S^n(M,S)(-n)$.  In this case $R$ is called {\em level} if $\omega_R$ is generated in a single degree.  The minimal number of generators of $\omega_R$ is called the {\em type} of $R$ and denoted throughout this paper by $\type(R)$.  If $\type(R) = 1$, i.e.~if $\omega_R$ is isomorphic to $R$, up to a shift in the grading, then $R$ is Gorenstein.  Equivalently, $R$ is Artinian and Gorenstein if it is injective as an $R$-module.  Gorenstein ideals have symmetric Betti tables and thus Gorenstein rings have palindromic $h$-vectors.  There are many examples of Gorenstein rings of interest in algebraic geometry, such as coordinate rings of many canonical curves, rings of invariants, and monomial curves.  We refer the reader to \cite{Hu} for a history of Gorenstein rings.
  
  Following \cite{MSS1}, we say that $R$ is {\em superlevel} if $R$ is level and $\omega_R$ is linearly presented over $R$.  Note that for $R$ to be superlevel, it is sufficient for $R$ to be level and $\omega_R$ be linearly presented over $S$.  The {\em idealization} (sometimes called the Nagata idealization or  trivial extension) of $R$ with respect to its canonical module is the ring
  \[
  \tilde{R} := R \ltimes \omega_R(- \reg(R) - 1),
  \]  with multiplication given by $(r_1, z_1) \cdot (r_2,z_2) = (r_1r_2, r_1z_2 + r_2z_1)$.  When $R$ is level, $\tilde{R}$ is a standard graded ring.  It is well-known that when $R$ is Artinian, $\tilde{R}$ is Artinian and Gorenstein; see \cite[proof of Theorem 3.3.6]{BH} or \cite[Theorem 2.76]{HMMNWW}.  
  Mastroeni, Schenck, and Stillman observed the following:
  
  \begin{thm}[{\cite[Proposition 2.2, Lemma 2.3, Theorem 2.5]{MSS1}}]
  \label{idealize} 
  Let $R = S/I$ be a standard graded, Artinian $k$-algebra.  
  \begin{enumerate}
\item  If $R$ is level, then $\tilde{R}$ is a standard graded, Artinian, Gorenstein $k$-algebra.   In this case 
\[\codim(\tilde{R}) = \codim(R) + \mathrm{type}(R) \quad \text{ 
 and } \quad
 \reg(\tilde{R}) = \reg(R) + 1.\]
\item If $R$ is quadratic and superlevel, then $\tilde{R}$ is quadratic.  
\item If $R$ is not Koszul, then $\tilde{R}$ is not Koszul.
\item $\tilde{R} \cong S[y_1,\ldots,y_t]/((I) + \mathcal{L} + (y_1,\ldots,y_t)^2)$, where $t = \type(R)$ and \[\mathcal{L} = \left.\left( \sum_{i = 1}^t f_i y_i\,\right| \,(f_1,\ldots,f_t) \in \Syz_1^S(\omega_R)\right).\]
\end{enumerate}
\end{thm}

\noindent Thus idealizations of superlevel, Artinian, quadratic algebras are a convenient way of constructing quadratic Gorenstein algebras.  All three of the constructions in this paper use this idea.

 \section{A non-Koszul, quadratic, Gorenstein ring with codimension 8 and regularity 3}\label{s3}
 
A construction of Matsuda \cite{Mat} shows that not every quadratic, Gorenstein ideal is Koszul.  Matsuda's example had regularity $4$.   Conca, Rossi, and Valla showed that every quadratic, Gorenstein algebra with regularity 2 was Koszul \cite[Proposition 2.12]{CRV} and asked asked whether every such algebra with regularity $3$ was Koszul \cite[Question 6.10]{CRV}.   It is known that all quadratic, Gorenstein algebras of regularity $3$ and codimension at most $5$ are Koszul \cite{Caviglia, CRV}.   Mastroeni, Schenck, and Stillman \cite{MSS1} constructed counterexamples in all codimensions $c \ge 9$.  They then posed the following question:

\begin{qst}[{\cite[Question 1.3]{MSS1}}]\label{rc} For which positive integers $c$ and $r$ is every quadratic Gorenstein
ring $R$ with $\codim(R) = c$ and $\reg(R) = r$ Koszul?
\end{qst}

\noindent In a second paper \cite{MSS2} Mastroeni, Schenck and Stillman settle this question for all ordered pairs $(r,c)$ except for $(3,6)$, $(3,7)$, and $(3,8)$.  In this section, we show that the answer to Question~\ref{rc} is negative for $(r,c) = (3,8)$.
We construct our example by starting with one of the $4$-variable quadratic algebras that Roos compiled in \cite{Roos}.  In particular, he constructed a non-Koszul, Artinian, quadratic, superlevel $k$-algebra of regularity $2$.  Applying Theorem~\ref{idealize} to it, we obtain the following result.
 
 \begin{thm} Let $k$ be a field of characteristic $0$ and let $S = k[u,x,y,z]$.  Let $I = (x^2+yz+u^2,xu,x^2+xy,xz+yu,zu+u^2,y^2+z^2)$.  Then $R = S/I$ is  non-Koszul, Artinian, superlevel, with $\reg(R) = 2$ and $\type(R) = 4$.  Consequently, its idealization $\tilde{R} = R \ltimes \omega_R(-5)$ is a non-Koszul, quadratic, Gorenstein, Artinian, graded $k$-algebra with $\reg(\tilde{R}) = 3$ and $\codim(\tilde{R}) = 8$.
 \end{thm}
 
 \begin{proof} That $S/I$ is not Koszul follows from computations done by Roos \cite{Roos}.  A Macaulay2 \cite{M2} calculation shows that $S/I$ has graded Betti table
 \[
 \begin{tabular}{r|ccccc}
      &0&1&2&3&4\\ \hline
      \text{0:}&1&\text{-}&\text{-}&\text{-}&\text{-}\\\text{1:}&\text{-}&6&4&\text{-}&\text{-}\\\text{2:}&\text{-}&\text{-}&9&12&4.\\\end{tabular}
 \]
 In particular, $S/I$ is superlevel and $\type(R) = 4$.  Therefore by Theorem~\ref{idealize}, $\tilde{R} = R \ltimes \omega_R(1)$ is a non-Koszul, quadratic, Gorenstein, standard graded $k$-algebra with $\reg(\tilde{R}) = 3$ and $\codim(\tilde{R}) = 8$.
 \end{proof}
 
 %Computations in Macaulay2 yield that the Betti table of $\tilde{R}$  has the following form
 % \[
% \begin{tabular}{r|ccccccccc}
%      &0&1&2&3&4&5&6&7&8\\
%      \hline
%      \text{0:}&1&\text{-}&\text{-}&\text{-}&\text{-}&\text{-}&\text{-}&\text{-}&\text{-}\\\text{1:}&\text{-}&28&105&171&142&58&9&\text{-}&\text{-}\\\text{2:}&\text{-}&\text{-}&9&58&142&171&105&28&\text{-}\\\text{3:}&\text{-}&\text{-}&\text{-}&\text{-}&\text{-}&\text{-}&\text{-}&\text{-}&1\\
%       \end{tabular}
% \]
 
 \noindent The $h$-vector of $\tilde{R}$ is $(1,8,8,1)$.  The above example comes from \cite[Example 57, Table 5]{Roos}.  Other examples can be constructed from \cite[Examples 55 and 56, Table 5]{Roos}.
 
 It is natural to ask whether one can use the idealization of a non-Koszul, superlevel, Artinian algebra $R$ to settle the two remaining cases $(r,c) = (3,6)$ and $(3,7)$.  We show next that this is impossible.  Since the codimension of the resulting idealization $\tilde{R}$ is $\mathrm{codim}(R) + \type(R)$, we would need to find a quadratic, superlevel, Artinian algebra with $\codim(R) + \type(R) \le 7$.  This is impossible in view of the following result.
  
\begin{prop} If $R$ is a level, Artinian, quadratic non-Koszul algebra with $\reg(R) = 2$, then $\codim(R) + \type(R) \ge 8$. Hence $\tilde{R} = R \ltimes \omega_R$ has codimension at least $8$.
\end{prop}
\begin{proof}
Assume towards a contradiction that $\codim(R) + \type(R) \le 7$.
Since $R$ is level with $\reg(R)=2$, $\type(R) =\hilb_R(2)$.  If $\hilb_R(2) \le 2$, then $R$ is Koszul by \cite[4.8]{Backelin}. 
 If $\hilb_R(2) =3 $, we may appeal to \cite[Theorem 1.1]{Conca} to conclude that $R$ is Koszul.  If $\hilb_R(2) \ge 4$, then  $c=\codim(R) \le 3$, which is impossible since $\hilb_R(2) \le \binom{c+1}{2} - 3 \le 3$ if $R$ is Artinian and quadratic.   
\end{proof}

\section{Subadditivity fails for Gorenstein ideals}\label{s4}

To place our second result in context, we begin by showing that homogeneous complete intersection rings $R$ enjoy the subadditivity property, that is  $t_a(R) + t_b(R) \ge t_{a+b}(R)$ for all $a, b \ge 1$.
\begin{prop} 
\label{prop:CIsub}
Subadditivity holds for homogeneous complete intersections.
\end{prop}

\begin{proof} Let $I = (f_1,\ldots,f_c)$ be a homogeneous complete intersection ideal with $\deg(f_i) = d_i$.  Then $S/I$ is resolved by a Koszul complex.  We order the generators so that $d_1 \ge d_2 \ge \cdots \ge d_c$.  By the construction of the Koszul complex, it follows that $t_i(S/I) = \sum_{j = 1}^i d_j$.  Hence for any positive integers $a,b$ with $a+b \le c$ we have
\[t_{a+b}(S/I) = \sum_{j = 1}^{a+b} d_j = \sum_{j = 1}^a d_j + \sum_{j = a+1}^{a+b} d_j \le \sum_{j = 1}^a d_j + \sum_{j = 1}^b d_j = t_a(S/I) + t_b(S/I).\]
\end{proof}

On the other hand, subadditivity fails in general, even for Cohen-Macaulay ideals cf.~\cite[Example 4.4]{EHU}.  Note that to study subadditivity for Cohen-Macaulay, and in particular, for Gorenstein rings, it suffices to consider Artinian rings $R = S/I$, since the graded Betti numbers of $R$ over $S$ are the same as those of $R/(\ell)$ over $S/(\ell)$ for any linear form $\ell \in S$ regular on $R$. 

The following Lemma is similar to an example due to Caviglia \cite[Example 4.4]{EHU}.  We use the notation $[-]:\Z\to\N, [x]=\max\{x,0\}$.

\begin{lem}\label{CMsubadd} Fix a natural number $m$ and a field $k$ with $\ch(k) = 0$ or $\ch(k)>2m+1$.  Let $S = k[x_1,\ldots,x_{2m}]$ and consider the ideals $C = (x_1^2,\ldots,x_{2m}^2)$ and $I = C + \left((x_1 + \cdots + x_{2m})^2\right)$.  Then for $m \ge 2$, $R := S/I$ is an Artinian, quadratic algebra that has the following properties.
\begin{enumerate}
\item The Hilbert function of $R$ is $\hilb_R(i) = \left [ \binom{2m}{i} - \binom{2m}{i-2}\right ]$.
\item $\reg(R) = m$.
\item $\beta_{2,m+2}^S(R) \neq 0$, and moreover, $t_2(R) = m+2$.
\item $R$ is superlevel.
\item $R$ is not Koszul.
\end{enumerate}
\end{lem}

\begin{proof} Tensoring with a field extension of $k$, if necessary, one may assume that $k$ is infinite. 
Set $L = C:I$ and note that $L$ is directly linked to $I$.  
Let $\ell = x_1  + \cdots + x_{2m}$ and consider the  homomorphism 
$\mu:(S/C)\to (S/C)$, where $\mu(x)=\ell^2x$ for which $\ker(\mu)=L/C$ and $\coker(\mu)=S/I$.
Since $\ell^2$ is a strong Lefschetz element for $S/C$ (see \cite[Theorem 3.35]{Lef} and the references therein for the characteristic zero case and \cite[Theorem 7.2]{Cook} for positive characteristics),  the $k$-linear functions $\mu_i:(S/C)_i \overset{\ell^2}{\rightarrow} (S/C)_{i+2}$, obtained  by restricting $\mu$ to each of the graded components of $S/C$, are injective for $i \le m-1$ and surjective for $i \ge m - 1$.  It follows that the Hilbert function of $R = S/I$ is 
\[\hilb_R(i) = \left[ \dim_k\left((S/C)_i\right) - \dim_k\left((S/C)_{i-2}\right) \right] =  \left [ \binom{2m}{i} - \binom{2m}{i-2} \right ].\]
  In particular, $R$ is Artinian with $\hilb_R(m) \neq 0$, while $\hilb_R(m+1) = 0$ and so $\reg(R) = m$.

Moreover, the injectivity of the maps $\mu_i$ above shows $\dim_k(L_i/C_i) =0$ for $i\leq m-1$, so $L$ has no minimal generators below degree $m$ besides the quadratic generators of $C$. The non-injectivity of $\mu_m$ shows that $L$ has minimal generators in degree $m$.

Consider the graded short exact sequence
\[0 \to S/L(-2) \xto{\ell^2} S/C \xto{} S/I \to 0.\]
Since $\beta_{1,m+2}(S/C) = 0$ and $\beta_{1,m+2}(S/L(-2)) = \beta_{1,m}(S/L) \neq 0$, it follows from the long exact sequence of $\tor$ that $\beta_{2,m+2}(S/I) \neq 0$.  Since the only minimal generators for both $C$ and $L$ below degree $m$ are those in the complete intersection $C$, it follows that $\beta_{i,j}(S/C) = 0$ for $j < 2i$ and 
\[
\beta_{i,j}(S/L(-2)) = \beta_{i,j+2}(S/L) = 0 \text{ for } j < \min\{m-2+i,2i-2\}.
\]
 Again by the long exact sequence of $\tor$ we obtain that $\beta_{i,j}(S/I) = 0$ for $j < \min\{2i, m+i\}$.

On the other hand, since $R = S/I$ is Artinian and its Hilbert function satisfies
\[
\hilb_R(i)=
\begin{cases}
 \left[ \binom{2m}{i} - \binom{2m}{i-2} \right] = 0 & \text{ for }i > m \text{ and } \\
 \binom{2m}{m} - \binom{2m}{m-2} \neq 0,  &\text{ for }i = m.  
\end{cases}
\]
 it follows that $\reg(R) = m$; so the vanishing of Betti numbers $\beta_{i,j}(R) = 0$ for $j < \min\{2i, m+i\}$ forces $\beta_{i,m + i}(R) \neq 0$ for $m < i \le 2m$.  In particular, $R$ is superlevel.

Finally to see that $R = S/I$ is not Koszul, note that if $m \ge 3$, then $t_2(R) = m+2 > 4 = t_1(R) + t_1(R)$.  If $R$ were Koszul, then this would contradict \cite[Theorem 6.2]{ACI2}. For $m=2$ the non Koszul property can be checked by direct computation in Macaulay2 \cite{M2}.
\end{proof}

We now construct a family of quadratic, Artinian, Gorenstein graded rings that have several bad properties. 

\begin{thm}\label{subgor} 
Fix an integer $m \ge 2$ and let $k$ be a field with $\ch(k) = 0$ or $\ch(k)>2m+1$.  There exists a quadratic, Artinian, Gorenstein, graded $k$-algebra $A$ with the following properties
\begin{enumerate}
\item $\codim(A) = 2m + \left [ \binom{2m}{m} - \binom{2m}{m-2}\right ]$.
\item $A$ is not Koszul.
\item $t_2(A) = m+2$.
\item $\reg(A) = m + 1$.
\item The Hilbert function of $A$ is 
\[\mathrm{HF}_A(i)=\left [ \binom{2m}{i} - \binom{2m}{i-2}\right ] + \left [ \binom{2m}{m-i+1} - \binom{2m}{m - i - 1}\right ].\]
\end{enumerate}
In particular, 
\begin{itemize}
\item $A$ does not satisfy the subadditivity property if $m \ge 3$, and 
\item $A$ has a non-unimodal Hilbert function if $m \ge 7$ and thus does not satisfy the weak or strong Lefschetz properties.
\end{itemize}
\end{thm}

\begin{proof} Set $A=\tilde{R}=R\ltimes \omega_R(-m-1)$, where $R=S/I$ is the Artinian algebra introduced in the statement of Lemma \ref{CMsubadd} for $S = k[x_1,\ldots,x_{2m}]$. That $A$ is quadratic, Artinian, non-Koszul, and Gorenstein with the claimed regularity, codimension and Hilbert function follows from Theorem~\ref{idealize} and Lemma~\ref{CMsubadd}.  

Also by Theorem~\ref{idealize} we can write a presentation for $A$ as 
\[\ A=T/M \text{ with } M = ((I) + \mathcal{L} + (y_1,\ldots,y_t)^2),\]
 where $t = \type(R)=\left [ \binom{2m}{m} - \binom{2m}{m-2}\right ]$ and $T = S[y_1,\ldots,y_t]$.  
 It will be useful at this time to view $A$ as a bigraded ring with respect to the grading obtained by assigning degree $(1,0)$ to the variables of $S$  and degree $(0,1)$ to the variables $y_i$. The short exact sequence of bigraded T-modules
 \[ 0\to \mathcal{L} + (y_1,\ldots,y_t)^2 \to T/IT\to A\to 0\] gives rise to a long exact sequence containing the fragment
 \[
 \cdots \to \tor_2^T(\mathcal{L} + (y_1,\ldots,y_t)^2,k)\to \tor_2^T(T/IT,k) \to \tor_2^T(A,k) \to \cdots .
 \]
Since $(\mathcal{L} + (y_1,\ldots,y_t)^2)_{(*,0)}=0$, also $(\tor_2^T(\mathcal{L} + (y_1,\ldots,y_t)^2,k))_{(*,0)}=0$ for all $*\in \N$. By contrast, since $y_1,\ldots,y_t$ is a regular sequence on $T/IT$, it follows that $\tor_2^T(T/IT,k)=\tor_2^S(S/I,k) \tensor_S T$ is concentrated in degrees $(*,0)$. It follows that the long exact sequence above splits inducing an injection $ \tor_2^T(T/IT,k) \into \tor_2^T(A,k)$.
 Hence $t_2(A) \ge t_2(S/I) = m$.  
 
 On the other hand, since $\reg(S/I) = m$, we get from Theorem~\ref{idealize} that $\reg(A) = m+1$; i.e. $t_i(A) \le m+1+i$ for all $i \ge 0$.  Moreover, since $A$ is Gorenstein and quadratic, the symmetry of the Betti table of $A$ over $T$ forces $t_i(A) \le m+i$ for all $i < 2m + t$.  In particular, this implies that $t_2(A) = m + 2 > 2 + 2 = t_1(A) + t_1(A)$ and thus $A$ fails the subadditivity property when $m \ge 3$.

Finally, assume that $m \ge 7$.  If $m = 7$, then $\mathrm{HF}_A(3) = 1988 < 2092 = \mathrm{HF}_A(2)$.  If $m = 8$, then $\mathrm{HF}_A(3) = 6732 < 7191 = \mathrm{HF}_A(2)$.  If $m=9$ then $\mathrm{HF}_A(3) = 24054 < 25346 = \mathrm{HF}_A(2)$. We now show 
$\mathrm{HF}_A(1) > \mathrm{HF}_A(\left \lfloor \frac{m}{2}\right \rfloor)$ for $m \ge 10$. By (5) we have
\[\mathrm{HF}_A(1)=2m+ \binom{2m}{m} - \binom{2m}{m -2}=2m+\frac{2\cdot (2m+1)!}{m!(m+2)!}\]
and thus it suffices  to show that $\frac{2\cdot (2m+1)!}{m!(m+2)!}/\mathrm{HF}_A\left(\left \lfloor \frac{m}{2}\right \rfloor \right)>1$.
Consider first the case $m=2n$, whence by use of Pascal's formula one computes
\begin{eqnarray*}
\mathrm{HF}_A\left(n\right) &= &\binom{4n}{n} - \binom{4n}{n-2} + \binom{4n}{n+1} - \binom{4n}{n-1}\\
&=&\binom{4n+1}{n+1} -\binom{4n+1}{n-1} =\frac{(4n+1)!\cdot 2(2n+1)^2}{(3n+2)!(n+1)!}.
\end{eqnarray*}
We deduce the desired inequality by considering the function
\begin{eqnarray*}
\frac{\frac{2\cdot (2m+1)!}{m!(m+2)!}}{\mathrm{HF}_A\left(n\right)} &=& \frac{2\cdot(4n+1)!}{(2n)!(2n+2)!}\cdot \frac{(3n+2)!(n+1)!}{(4n+1)!\cdot 2(2n+1)^2} \\
&=&\frac{(2n+3)(2n+4)\cdots (3n+1)(3n+2)}{(n+2)(n+3)\cdots(2n)(2n+1)^2}\\
&=&\prod_{i=n+2}^{2n+1} \left(1+\frac{n+1}{i}\right)\cdot\frac{1}{2n+1}\\
&>& \left(1+\frac{n+1}{2n+1}\right)^n\cdot\frac{1}{2n+1}\geq \left(\frac{17}{11}\right)^n\cdot\frac{1}{2n+1}.
\end{eqnarray*}

Clearly the last  function above attains arbitrarily large values asymptotically and one can check that its values surpass $1$ for $n\geq 6$. In the remaining case, $n=5$, the claim $58806=\mathrm{HF}_A(1)> \mathrm{HF}_A\left(n\right)=48279$ can be checked by direct computation.  The case when $m$ is odd is similar and we omit the details.
%ODD CASE ARGUMENT COMMENTED OUT FOR NOW
 %Lastly, for the case $m=2n+1$ we see that
%\begin{eqnarray*}
%\mathrm{HF}_A\left(n\right) &= &\binom{4n+2}{n} - \binom{4n+2}{n-2} + \binom{4n+2}{n+2} - \binom{4n+2}{n}\\
%&=&\frac{(4n+2)!\cdot4(n+1)(4n+3)(5n^2+8n+2)}{(n+2)!(3n+4)!}.
%\end{eqnarray*}
%and we estimate for $n\geq 5$
%\begin{eqnarray*}
%\frac{\frac{2\cdot (2m+1)!}{m!(m+2)!}}{\mathrm{HF}_A\left(n\right)} 
%&=& \frac{2\cdot(4n+3)!}{(2n+3)!(2n+1)!}\cdot \frac{(n+2)!(3n+4)!}{(4n+3)!\cdot4(n+1)(5n^2+8n+2)}\\
%&=&\frac{(2n+4)(2n+3)\cdots(3n+3)}{(n+3)(n+4)\cdots(2n+2)}\cdot\frac{3n+4}{5n^2+8n+2}\\
%&=&\prod_{i=n+3}^{2n+2}\left(1+\frac{n+1}{i}\right)\cdot\frac{3n+4}{5n^2+8n+2}\\
%&>&\left(1+\frac{n+1}{2n+2}\right)^n\frac{3n+4}{5n^2+8n+2}\geq \left(\frac{3}{2}\right)^n\frac{3n+4}{5n^2+8n+2}.
%\end{eqnarray*}
%Clearly the last  function above attains arbitrarily large values asymptotically and one can check that its values surpass $1$ for $n\geq 6$. In the remaining case $n=5$ we have that $208034=\mathrm{HF}_A(1)> \mathrm{HF}_A\left(n\right)=169004$  by direct computation. 
\end{proof}

That the family of Artinian algebras in Theorem \ref{subgor} fails to satisfy the strong Lefschetz property can be deduced from \cite[Proposition 2.1]{Gondim}, however the stronger statement regarding the failure of the weak Lefschetz property is to our knowledge new. 

\begin{eg} Take $m = 3$ and consider $S = k[x_1,\ldots,x_6]$.  Applying Lemma~\ref{CMsubadd}, we set $C = (x_1^2,\ldots,x_{6}^2)$ and $I = C + (x_1 + \cdots + x_6)^2$.  Then $R = S/I$ has the following Betti table over $S$:
\[
\begin{tabular}{r|ccccccc}
     &0&1&2&3&4&5&6\\
    \hline
    \text{0:}&1&\text{-}&\text{-}&\text{-}&\text{-}&\text{-}&\text{-}\\\text{1:}&\text{-}&7&\text{-}&\text{-}&\text{-}&\text{-}&\text{-}\\\text{2:}&\text{-}&\text{-}&21&\text{-}&\text{-}&\text{-}&\text{-}\\\text{3:}&\text{-}&\text{-}&14&105&132&70&14\\\end{tabular}
    \]

\noindent In particular, $t_1(S/I) = 2$ and $t_2(S/I) = 5$, so $R$ is an Artinian algebra that fails subadditivity; i.e. $t_2(R) > t_1(R) + t_1(R)$.  Moreover, $R$ is superlevel with $\reg(R) = 3$ and $h$-vector $(1,6,14,14)$.

Now we consider $\tilde{R} = R \ltimes \omega_R(-4)$.  By Theorem~\ref{idealize}, $\tilde{R}$ is Artinian, Gorenstein with $h$-vector $(1,20,28,20,1)$ and $\reg(\tilde{R}) = 4$.  As $\tilde{R}$ is quadratic, $t_1(\tilde{R}) = 2$ while $t_2(\tilde{R}) = 5$ by Theorem~\ref{subgor}.  Thus $\tilde{R}$ is an Artinian, Gorenstein algebra for which subadditivity fails.  

It is worth noting that while we know the Hilbert function of $\tilde{R}$ from Theorem \ref{subgor}, the full Betti table of $\tilde{R}$, as a quotient of a polynomial ring in 20 variables, is not so clear.  It would be very interesting to have a full description of the resolution of the idealization $\tilde{R}$ in terms that of $R$.  %However, we can at least say that the Betti table of $\tilde{R}$, resolved over a $20$ variable polynomial ring, has the following form:
%\[
%\begin{tabular}{r|ccccccccc}
%     &0&1&2&3&$\cdots$&17&18&19&20\\
%    \hline
%    \text{0:}&1&\text{-}&\text{-}&\text{-}&\text{-}&\text{-}&\text{-}&\text{-}&\text{-}\\\text{1:}&\text{-}&182&2120&$\ast$&$\cdots$&$\ast$&14&\text{-}&\text{-}\\
%    \text{2:}&\text{-}&\text{-}&21&?&$\cdots$&?&21&\text{-}&\text{-}\\
%    \text{3:}&\text{-}&\text{-}&14&$\ast$&$\cdots$&$\ast$&2120&182&\text{-}\\
%     \text{4:}&\text{-}&\text{-}&\text{-}&\text{-}&$\cdots$&\text{-}&\text{-}&\text{-}&1,\\
%    \end{tabular}
%    \]
%where `\text{-}' denotes a 0 entry, `$\ast$' denotes a nonzero entry and `?' denotes a possibly nonzero entry.  %We work out a few of the Betti numbers based on the Hilbert function of $\tilde{R}$ from Theorem \ref{subgor}. There are 182 quadric generators of the defining ideal of $\tilde{R}$.  Because $\dim_k(\tilde{R}_3) = 20$, it follows that 
%\[\beta_{2,3}(\tilde{R}) = 182 \times 20 - \binom{20+3-1}{3} + 20 = 2120.\]
%  A Macaulay2 computation gives $\beta_{2,4}(\tilde{R}) = 21$ and $\beta_{2,5}(\tilde{R}) = 14$.  The remaining nonzero entries are forced by the symmetry of the Betti table, the fact that $t_2(\tilde{R}) = 5$, and the fact that the $t_i(\tilde{R})$ form a strictly increasing sequence \cite{M}.
\end{eg}

\begin{eg} When $m = 7$, the Gorenstein $k$-algebra $\tilde{R}$ from Theorem~\ref{subgor} has  {\em non-unimodal} $h$-vector 
%given by the sum of the $h$-vectors of $R$ and $\omega_R(-8)$:
%\[(1,14,90,350, 910, 1638, 2002, 1430) + (0,1430,2002,1638,910,350,90,14,1)\]
%\[ =
$ (1,1444,2092,1958,1820,1958,2092,1444,1).$
%\]  
Thus $\tilde{R}$ is a quadratic Gorenstein algebra with $\codim(\tilde{R}) = 1444$ and $\reg(\tilde{R}) = 8$ and for which both the weak and strong Lefshetz properties fail.
The first example of a (non-quadratic) Gorenstein algebra with non-unimodal $h$-vector was famously constructed by Stanley \cite{S}. \footnote{Stanley's construction was also given via idealization for the ring $A = k[x,y,z]/(x,y,z)^4$.  Since $A$ has $h$-vector $(1,3,6,10)$, its idealization $\tilde{A} = A \ltimes \omega_A$ has $h$-vector $(1,13,12,13,1)$; however, $\tilde{A}$ is clearly not quadratic.}
\end{eg}

\section{Quadratic non-Koszul Gorenstein algebras with linear resolutions of arbitrarily high order}     \label{s5}

Let $k$ be a field of characteristic $0$.  In \cite[Thereom 1']{Roos2} Roos gave examples of graded Artinian non-Koszul quadratic $k$-algebras $A_\alpha$ for integers $\alpha \ge 2$ such that $k$ has a linear resolution for precisely the first $\alpha$ steps before a minimal non-linear syzygy.  Here we note that these algebras are superlevel and that their Gorenstein idealizations have the same property.  This removes any hope of a `finite test' for the Koszul property in the context of quadratic Gorenstein algebras.

To state the result, we first recall the definition of the graded Poincare series of a module.  Fix a graded module $M$ over a graded ring $R$.  The  generating function for the bigraded $k$-vector space $\tor_\ast^R(M,k)_\ast$ is $P_R^M(x,y) := \sum_{i,j} \beta_{x,y}^R(M) x^i y^j$.  Thus the \textit{graded Poincare series} of $R$ is $P_R^k(x,y)$, which encodes the resolution of $k$ over $R$.  

\begin{thm} Let $k$ be a field of characteristic $0$ and fix a positive integer $\alpha \ge 2$.  Let $S = k[u,v,w,x,y,z]$ and 
\[I = (x^2,xy,y^2,yz,z^2,zu,u^2,uv,v^2,vw,w^2,xz+\alpha zw - uw, zw + xu + (\alpha - 2) uw).\]
Then the idealization $\tilde{R}$ of $R = S/I$ is quadratic, Gorenstein, non-Koszul, and $k$ has linear resolution for precisely $\alpha$ steps in the resolution over $\tilde{R}$.
\end{thm}

\begin{proof} By \cite[Theorem 1']{Roos2}, $R$ has Hilbert series $\mathrm{HS}_R(t) = 1 + 6t + 8t^2$, whence  $\reg(R) = 2$, and $\type(R) = 8$, and it is easy to check with Macaulay2 \cite{M2} that $R$ is superlevel.  Thus $\tilde{R} = R \ltimes \omega_R(-3)$ is Artinian, quadratic, and Gorenstein with $\reg(\tilde{R}) = 3$ and $\codim(\tilde{R}) = 6+8 = 14$.  

We need only argue that the resolution of $k$ over $\tilde{R}$ is linear for exactly $\alpha$ steps. To achieve this, we recall the details of the construction of $R$ from \cite{Roos2}. It is shown therein that $R$ itself is an idealization $R=A\ltimes M(-1)$, where $A=k[x,y,u,v]/(x^2,xy,y^2,v^2,vw,w^2)= B \otimes_k C$, $B = k[x,y]/(x,y)^2$, and $C =  k[v,w]/(v,w)^2$ and $M$ is an $A$-module with Hilbert series $\mathrm{HS}_M(t)=2+4t$.
By a result of Gulliksen \cite[Theorem 2]{Gulliksen} combined with the fact that $M$ and $\omega_R$ are linearly presented, the relevant graded Poincar\'e series are related by 
\begin{eqnarray}
P_{\tilde{R}}^k(x,y) &=& P_R^k(x,y)(1 - xy P_R^{\omega_R(-2)}(x,y))^{-1}\\
P_{R}^{\omega_R}(x,y) &=& P_A^{\omega_R}(x,y)(1 - xy P_A^M(x,y))^{-1}.
\end{eqnarray}
Since \cite{Roos2} shows that both the resolution of $k$ over $R$ and the resolution of $M$ over $A$ are linear for exactly $\alpha$ steps, it suffices to show that the resolution of $\omega_R$ over $A$ is linear. 

By \cite[proof of Claim 2]{Loan} there is an isomorphism of $R$-modules $\omega_R(-2) \cong \omega_M(-1) \ltimes \omega_A(-2)$, with the $R=A\ltimes M(-1)$-module structure given by $(a,m)\cdot(s,t)=(as,at+s(m))$, where we view $s\in \hom_A(M,\omega_A)\cong  \omega_M$.
This induces an isomorphism of $A$-modules $ \omega_R(-2)  \cong  \omega_M(-1) \oplus \omega_A(-2)$, and we note that both $\omega_M(-1)$ and $\omega_A(-2)$ are generated in degree $0$.  
Thus it suffices to show that both $\omega_A(-2)$ and $\omega_M(-1)$ have linear resolutions over $A$.
The former module decomposes as a tensor product $\omega_A(-2)\cong\omega_B(-1)\otimes_k \omega_C(-1)$.  Thus the minimal free resolution of $\omega_A(-2)$ over $A$ is in turn the tensor product of the resolutions of $\omega_B(-1)$ over $B$ and $\omega_C(-1)$ over $C$. Since the rings $B$ and $C$ contain no elements of degree greater than one, the differentials in the resolutions of $\omega_B(-1)$ and $\omega_C(-1)$ are linear. As $B$ and $C$ are level of regularity 1, $\omega_B(-1)$ and $\omega_C(-1)$ are each generated in degree zero, hence  their resolutions are linear and so is the resolution of $\omega_A(-2)$ over $A$. 

Finally, in order to analyze the resolution of $\omega_M(-1)$ over $A$ one must dig deeper into the structure of the module $M$. Since $A$ is Artinian, \cite[Proposition 21.1]{Eis} shows that $\omega_M(-1) \cong\hom_k(M(+1),k)$ as an $A$-module. Consequently the Hilbert series of this module is $\mathrm{HS}_{\omega_M(-1)}(t)=4+2t$. Fix $k$-bases  $\{f_1^*, f_2^*, f_3^*, f_4^*\}$ for $ ({\omega_M(-1)})_0$ and $\{e_1^*, e_2^*\}$ for $({\omega_M}(-1))_1$, where $e_i, f_i$ refer to the dual elements in $M_0$ and $M_1$ respectively.
The $A$-module structures of $M$ and $\omega_M$ can be completely described by the following dual tables, the first of which can be deduced from \cite[Equation (3)]{Roos}
\[
\begin{small}
\begin{tabular}{c|cccc}
            & $f_1$ & $f_2$ & $f_3$ & $f_4$\\
            \hline
$e_1$ & $v$ & $w$ & $x+\alpha w$ & $0$\\
$e_2$ & $0$ & $(\alpha-2)w-x$ & $w$ & $y$
\end{tabular}
\qquad
\begin{tabular}{c|cccc}
            & $f_1^*$ & $f_2^*$ & $f_3^*$ & $f_4^*$\\
            \hline
$e_1^*$ & $v$ & $w$ & $x+\alpha w$ & $0$\\
$e_2^*$ & $0$ & $(\alpha-2)w-x$ & $w$ & $y$
\end{tabular}
\end{small}
\]
The leftmost table should be interpreted to mean that, for example, $ve_1=f_1$ in $M$, while the rightmost table should be interpreted to mean that, dually, $vf_1^*=e_1^*$ in $\omega_M$.  (That for instance $ye_1 = 0$ is thus implicit in the table.)  Equipped with this information, we implement a strategy for determining $\tor^A(\omega_M(-1),k)$ inspired by Roos's original approach: we compute the homology of the tensor product of the bar resolutions for $k$ over $B$ and $C$ further tensored with $\omega_M$. Since $\omega_M$ is concentrated in only two degrees,  the only potential nonzero components of  $\tor_i^A(\omega_M(-1),k)$ are in degrees $i$ and $i+1$. Furthermore $\tor_i^A(\omega_M(-1),k)_{i+1}$ is the cokernel of
\[
(\omega_M(-1))_0\otimes_k  \left(\bigoplus_{q=0}^i B_1^{\otimes i-q}\otimes_k C_1^{\otimes i}\right)  \to (\omega_M(-1))_1\otimes_k \left(\bigoplus_{q=0}^{i-1} B_1^{\otimes i-q-1}\otimes_k C_1^{\otimes i}\right)
\]
\[
s\otimes b_1\otimes \cdots \otimes b_{i-q}\otimes c_1\otimes \cdots\otimes c_q \mapsto
\]
\[
s b_1\otimes \cdots \otimes b_{i-q}\otimes c_1\otimes \cdots\otimes c_q +(-1)^{i-q} sc_1\otimes b_1\otimes \cdots \otimes b_{i-q}\otimes c_2\otimes \cdots\otimes c_q.
\]
It is easily verified that this map is surjective as for arbitrary $\mu, \lambda\in  k, b_j\in B_1, c_j\in C_1$ we have 
\[
\lambda f_4^*\otimes y\otimes  b_2\otimes \cdots \otimes b_{i-q}\otimes c_1\otimes \cdots\otimes c_q +(-1)^{i-q}\mu f_1^*\otimes  b_2\otimes \cdots \otimes b_{i-q}\otimes v \otimes c_1 \cdots\otimes c_q\mapsto 
\]
\[
(\mu e_1^*+\lambda e_2^*) \otimes  b_2\otimes \cdots \otimes b_{i-q}\otimes c_1\otimes \cdots\otimes c_q.
\]
Thus $\tor_i^A(\omega_M(-1),k)_{i+1}=0$ for $i>0$, which concludes the proof.
\end{proof}

 % ------------------------------------------------------------------------

\subsection*{Acknowledgment}
We thank Matt Mastroeni, Hal Schenck, and Mike Stillman for useful conversations and bringing Question~\ref{rc} to our attention.

 \end{document}